\documentclass[12pt]{amsart}

\usepackage[utf8]{inputenc}
\usepackage[T1]{fontenc}

\usepackage{verbatim} 
\usepackage[matrix, arrow, cmtip]{xy}
\usepackage{tikz}
\usetikzlibrary{matrix, positioning, arrows.meta}

\usepackage{amssymb, enumerate, enumitem, longtable, float, graphicx, array}

\makeatletter
\@namedef{subjclassname@2010}{%
  \textup{2010} Mathematics Subject Classification}
\makeatother

\newtheorem{theorem}{Theorem}[section]
\newtheorem{lemma}[theorem]{Lemma}
\newtheorem{corollary}[theorem]{Corollary}
\newtheorem*{carrier theorem}{Carrier Theorem}
\theoremstyle{definition}
\newtheorem{definition}[theorem]{Definition}
\newtheorem{numbered_definition}[theorem]{Definition}

\theoremstyle{remark}

\newcommand{\B}[1]{{\mathcal{B}_{#1}}}
\newcommand{\F}[1]{{\mathcal{F}_{#1}}}

\newcommand{\G}[1]{{\mathcal{G}_{#1}}}
\newcommand{\U}[1]{{\mathcal{U}_{#1}}}
\newcommand{\V}[1]{{\mathcal{V}_{#1}}}
\newcommand{\W}[1]{{\mathcal{W}_{#1}}}

\renewcommand{\O}[1]{{\mathcal{O}_{#1}}}
\newcommand{\df}[1]{{\bf #1}}

\newcommand{\dsup}{{d_{\sup\nolimits}}}

\DeclareMathOperator{\diam}{diam}

\DeclareMathOperator{\bst}{bst}
\DeclareMathOperator{\ost}{ost}

\DeclareMathOperator{\mesh}{mesh}

\DeclareMathOperator{\Int}{Int}

\DeclareMathOperator{\supp}{supp}
\DeclareMathOperator{\dom}{dom}
\begin{document}

\baselineskip=17pt

\title{Local $k$-connectedness of an inverse limit of polyhedra}

\author{G.~C.~Bell}
\address[G.~C.~Bell]{Department of Mathematics and Statistics, University of North Carolina at Greensboro, Greensboro, NC 27412, USA}
\email{gcbell@uncg.edu}

\author[A.~Nagórko]{A.~Nagórko}
\address[A.~Nagórko]{Faculty of Mathematics, Informatics, and Mechanics, University of Warsaw, Banacha 2, 02-097 Warszawa, Poland}
\email{amn@mimuw.edu.pl}
\thanks{This research was supported by the NCN (Narodowe Centrum Nauki) grant no. 2011/01/D/ST1/04144.}

\date{}

\subjclass[2010]{Primary 55M15, 54C55; Secondary 55M10}

\keywords{local $k$-connectedness, N\"obeling space}

\begin{abstract}
  We provide an easily verifiable condition for local $k$-connectedness
    of an inverse limit of polyhedra.
\end{abstract}

\maketitle

\section{Introduction}

  The N\"obeling space characterization theorem~\cite{ageev2007a, ageev2007b, ageev2007c, levin2006, nagorkophd} states that if a space is strongly universal in the class of $n$-dimensional separable complete metric spaces and is $k$-connected and locally $k$-connected for each $k < n$, then it is homeomorphic to the $n$-dimensional N\"obeling space $N^{2n+1}_n$.

  In geometric group theory many spaces arise naturally as inverse limits of polyhedra (simplicial complexes endowed with the metric topology).
  In particular, boundaries at infinity of hyperbolic spaces can be expressed as such.
  Striking examples of applications of the N\"obeling space characterization theorem are proofs that the boundary at infinity of the curve complex of the $(n+5)$-punctured $2$-dimensional sphere is homeomorphic to the $n$-dimensional N\"obeling space $N^{2n+1}_n$~\cite{henselprzytycki2011, gabai2014}.

  In the present paper we prove a condition for local $k$-connectedness of an inverse limit of polyhedra that is easy to verify. It is designed to aid detection of local $k$-connectedness in many examples of spaces arising in geometric group theory. We prove the following theorem.

\begin{definition}
  Let $K$ and $L$ be simplicial complexes.
  We say that a map $p \colon K \to L$ is \df{$n$-regular} if it is quasi-simplicial (i.e. it is simplicial into the first barycentric subdivision $\beta L$ of $L$) and if for each simplex $\delta$ of $\beta L$ the inverse image $p^{-1}(\delta)$ has vanishing homotopy groups in dimensions less than $n$ (regardless of the choice of  basepoint).
\end{definition}

\begin{theorem}\label{thm:main theorem}
  Let
  \[
    X = \lim_{\longleftarrow}\big( K_1 \xleftarrow{p_1} K_2 \xleftarrow{p_2} \cdots \big).
  \]
  Assume that for each $i$ the following conditions are satisfied.
  \begin{enumerate}[label=(\Roman*)]
  \item $K_i$ is a locally finite-dimensional simplicial complex endowed with the metric topology.
  \item $p_i$ is a quasi-simplicial map that is surjective and $n$-regular.
  \end{enumerate}
  Then $X$ is locally $k$-connected for each $k < n$, and each short projection $\pi^m_i \colon K_m \to K_i$ and each long projection $\pi_i \colon X \to K_i$ is a weak $n$-homotopy equivalence (it induces homomorphisms on homotopy groups in dimensions less than $n$, regardless of the choice of basepoint).
\end{theorem}
  Note that there is no assumption of local finiteness of complexes in the statement of Theorem~\ref{thm:main theorem}.

  It is known that (under suitable assumptions) if the bonding maps in the inverse limit are $n$-soft, then the inverse limit is $k$-connected for $k < n$,~\cite{chigogidze1989}.
  The condition for local $k$-connectedness given in the present paper may be regarded as a combinatorial analog of this statement for inverse limits of polyhedra.
  Note that $n$-regular bonding maps need not be $n$-soft.
  
\section{Preliminaries}

In this section we set the basic definitions and reference known results that will be used in the later sections.

\subsection{Absolute extensors in dimension $n$}

\begin{definition}
  We say that a space $X$ is \df{$k$-connected} if each map $\varphi \colon S^k \to X$ from a $k$-dimensional sphere into $X$ is null-homotopic in $X$.
\end{definition}

\begin{definition}
  We say that a space $X$ is \df{locally $k$-connected} if for each point $x \in X$ and each open neighborhood $U \subset X$ of $x$ there exists an open neighborhood $V$ of $x$ such that each map $\varphi \colon S^k \to V$ from a $k$-dimensional sphere into $V$ is null-homotopic in $U$.
\end{definition}

\begin{definition}
  We say that a metric space $X$ is an \df{absolute neighborhood
    extensor in dimension~$n$} if every map into $X$
  from a closed subset $A$ of an $n$-dimensional metric space extends
  over an open neighborhood of $A$. The class of absolute neighborhood
  extensors in dimension $n$ is denoted by $ANE(n)$ and its elements
  are called $ANE(n)$-spaces.
\end{definition}

\begin{definition}
We say that a metric space $X$ is an \df{absolute extensor 
  in dimension~$n$} if every map into $X$ from a closed subset 
  of an $n$-dimensional metric space $Y$ extends over the 
  entire space $Y$.
The class of absolute extensors in dimension~$n$ is denoted by
  $AE(n)$ and its elements are called $AE(n)$-spaces.
\end{definition}

\begin{definition}\label{def:AE(C)}
  Let~$\mathcal{C}$ be a class of topological spaces. 
  We let $AE(\mathcal{C})$ denote the class of \df{absolute extensors for all spaces 
  from the class $\mathcal{C}$}. We write $AE(X)$ for $AE(\{X\})$.
\end{definition}

Absolute extensors and absolute neighborhood extensors in dimension $n$ were characterized by Dugundji in the following theorem.

\begin{theorem}[\cite{dugundji1958}]\label{thm:dugundji}
  Let $X$ be a metric space. Then, 
	\begin{enumerate}
		\item $X\in ANE(n)\iff X$ is locally $k$-connected for all $k< n$; and 
		\item $X\in AE(n)\iff X\in ANE(n)$ and $X$ is $k$-connected for all $k< n$.
	\end{enumerate}
\end{theorem}

\begin{lemma}[cf.~{\cite[Theorem 6.1.8]{Sakai_book}}]\label{lem:union is ae}
  Assume that $A_1 \subset A_2 \subset \ldots$ is a sequence of subsets of a metric space such that each $A_i$ is closed and for each $i$, $A_i \subset \Int A_{i+1}$. If for each $i$, $A_i$ is $AE(n)$, then $A = \bigcup_i A_i$ is $AE(n)$.
\end{lemma}

\subsection{Polyhedra}

For a simplicial complex, the underlying polyhedron has two topologies, the Whitehead (weak) topology and the metric topology.
The metric topology is the topology of point-wise convergence of barycentric coordinates~\cite[p. 100]{hu1965}.
The weak topology is metrizable if and only if the complex is locally finite, and
it coincides with the metric topology in this case.
Since we work in the metric category with complexes that are not locally finite, {\bf we always assume the metric topology on simplicial complexes}~\cite{lundell1969,hu1965}.

\begin{definition}
  Let $K$ be a simplicial complex.
  We let $\tau(K)$ denote the \df{triangulation of~$K$} (the set of simplices of $K$).
  We let $V(K)$ denote the \df{vertex set of $K$}.
  We let $\beta K$ denote the \df{barycentric subdivision of $K$} (i.e., the same space but with a finer triangulation $\tau(\beta K)$).
\end{definition}


\begin{definition}\label{def:polyhedron metric}
Let $K$ be a simplicial complex.
Let $\kappa > 0$.
Let $\ell_1(V(K))$ denote the Banach space $\{ x \in \mathbb{R}^{V(K)} \colon \sum_{v \in V(K)} | (x)_v | < \infty \}$,
  where $(x)_v$ denotes the $v$-th coordinate, 
  equipped with the standard $\| \cdot \|_1$ norm.
For $v \in v(K)$ define $e_v \colon V(K) \to \mathbb{R}$ by the formula $e_v(w) = 0$ for $w \neq v$ and $e_v(v) = \kappa$.
We embed each vertex $v$ of $K$ as $e_v \in \ell_1(V(K))$ and extend this embedding to $K$ to be affine on each simplex of $K$. 
We consider $K$ to be a subspace of $\ell_1(V(K))$.
We call the induced metric on $K$ the \df{metric of scale $\kappa$ on $K$}.
For $\kappa=1$ it is the standard metric, as defined in~\cite[p. 100]{hu1965}.
The topology induced by this metric is called \df{the metric topology}.
\end{definition}

\begin{definition}
  A \df{polyhedron} is a simplicial complex endowed with the metric topology.
\end{definition}

\begin{definition} Let $K$ and $L$ be polyhedra. We say that $K$ is a \df{full subpolyhedron} of $L$ if whenever the vertices $v_0,\ldots, v_n$ span a simplex in $K$ and each $v_i$ is a vertex in $L$, then the $v_i$ span a simplex in $L$.
\end{definition}

\begin{lemma}\label{lem:complex is ane}
  A locally finite-dimensional polyhedron is a complete metric $ANE(\infty)$-space.
\end{lemma}
\begin{proof}
  It is complete by~\cite[Lemma~11.5]{hu1965}. It is an $ANE(\infty)$-space by~\cite[Theorem 11.3]{hu1965}.
\end{proof}




\begin{definition}
  Let $K$ and $L$ be simplicial complexes and let $p \colon K \to L$.
  We say that $p$ is \df{quasi-simplicial} if it is a simplicial map into $\beta L$.
\end{definition}

\begin{lemma}[cf.~\cite{bellnagorko2013}]\label{lem:qs is lipschitz}
  Assume that $K$ and $L$ are polyhedra endowed with metrics of scale $\kappa$ and $\lambda$ respectively.
  If $p \colon K \to L$ is quasi-simplicial, then it is $\frac{\lambda}{2\kappa}$-Lipschitz.
\end{lemma}

\subsection{Weak $n$-homotopy}

\begin{definition}
  We say that a map is a \df{weak $n$-homotopy equivalence} if it induces isomorphisms on homotopy groups of dimensions less than $n$, regardless of the choice of basepoint.
\end{definition}




\begin{definition}
  Let $\F{}$ be a cover of a space $X$. We say that the two maps $f, g \colon Y \to X$ are \df{$\F{}$-close} if for each $y \in Y$ there exists $F \in \F{}$ such that $f(y), g(y) \in F$.
\end{definition}

\begin{definition} Let $\U{}$ be a cover of a space $X$. We say that two maps $f, g \colon Y \to X$ are \df{$\U{}$-homotopic} if there exists a homotopy $H \colon Y \times [0, 1] \to X$ whose paths refine $\U{}$, i.e. for each $y \in Y$ there exists $U \in \U{}$ such that $H(\{ y \} \times [0, 1])\subset U$.
\end{definition}


\subsection{Carrier Theorem}\label{sec:carrier theorem}

\begin{definition}
  Let~$\mathcal{C}$ be a class of topological spaces. 
  We say that a cover $\F{}$ of a topological space is a $\mathcal{C}$-cover, if for each $\mathcal{A} \subset \mathcal{F}$ the intersection $\bigcap \mathcal{A}$ is either empty or belongs to $\mathcal{C}$.
\end{definition}

\begin{definition}
  We say that a cover is \df{locally finite-dimensional} if its nerve is locally finite-dimensional.
\end{definition}
\begin{definition}
  A \df{carrier} is a function $C \colon \F{} \to \G{}$ from
  a cover~$\F{}$ of a space~$X$ into a collection~$\G{}$ of subsets of
  a topological space such that for each $\mathcal{A} \subset \F{}$ if
  $\bigcap_{A\in\mathcal{A}}A \neq \emptyset$, then $\bigcap_{A \in
  \mathcal{A}} C(A) \neq \emptyset$. We say that a map~$f$ is
  \df{carried by~$C$} if it is defined on a
  closed subset of~$X$ and $f(F) \subset C(F)$ for each $F \in \F{}$. Here we write $f(F)$ to mean $f(F\cap \dom(f)).$
\end{definition}

\begin{carrier theorem}[\cite{nagorko2007}]
  Assume that $C \colon \F{} \to \G{}$ is a carrier such that~$\F{}$ is a
  cover of a space~$X$ and~$\G{}$ is an $AE(X)$-cover of another space.
  If~$\F{}$ is closed, locally finite, and locally finite-dimensional, then each map
  carried by~$C$ extends to a map of the entire space~$X$, also carried
  by~$C$.
\end{carrier theorem}

\begin{corollary}[\cite{nagorko2007}]
  If $\F{}$ is a closed locally finite locally finite-dimensional $AE(n)$-cover of a space $Y$, 
  then any two $\F{}$-close maps from a metric space of dimension less than $n$ into $Y$ are $\F{}$-homotopic.
\end{corollary}


\subsection{Covers}
We regard covers as indexed collections of sets
and use the usual notation, $\F{} = \{ F_i \}_{i \in I}$,
  where $I$ denotes the indexing set.

\begin{numbered_definition}\label{def:stars}
  Let $K$ be a polyhedron. Let $L \subset K$ be a subcomplex of $K$.
  The \df{open star $\ost_K L$ of $L$ in $K$} is the complement of the union of all simplices of $K$ that do not intersect $L$:
  \[
    \ost_K L = K \setminus \bigcup \{ \delta \in \tau(K) \colon \delta \cap L = \emptyset \}\text{.}
  \]
  The \df{barycentric star $\bst_K L$ of $L$ in $K$} is the union of all simplices of $\beta K$ that intersect $L$:
  \[
    \bst_K L = \bigcup \{ \delta \in \tau(\beta K) \colon \delta \cap L \neq \emptyset \}\text{.}
  \]

  We let
  \[
    \O{K} = \{ \ost_K \{ v \} \colon v \in V(K) \}
  \]
  denote the \df{cover of $K$ by open stars of vertices}.
  We let
  \[
    \B{K} = \{ \bst_K \{ v \} \colon v \in V(K) \}
  \]
  denote a \df{cover of $K$ by barycentric stars of vertices}.
\end{numbered_definition}

\begin{lemma}[\cite{nagorko2013}]\label{lem:stars}
  Let $K$ be a polyhedron.
  The cover $\B{K}$ by barycentric stars of vertices
    is a closed locally finite $AE(\infty)$-cover of $K$. 
  Moreover, if~$K$ is locally finite-dimensional, then $\B{K}$ is locally finite-dimensional.
  The cover~$\O{K}$ by open stars of vertices is an open
    $AE(\infty)$-cover of $K$.
\end{lemma}

\begin{definition}
  Let
  \[
  X = \lim_{\longleftarrow} \left(   
  K_1\xleftarrow{p_1}K_2\xleftarrow{p_2}\cdots \right)
  \]
  Let $v(K_i)$ denote the set of vertices of $K_i$.
  We let
  \[
    \O{K_i} = \{ O_v = \ost_{K_i} v \}_{v \in v(K_i)}
  \]
  be the cover of $K_i$ by open stars of vertices of $K_i$ (see Definition~\ref{def:stars}) and
  \[
    \O{i} = \{ \pi^{-1}_i(\ost_{K_i} v) \}_{v \in v(K_i)}
  \]
  be the cover of $X$ by sets of threads that pass through elements of $\O{K_i}$.
\end{definition}

\begin{definition}
  Let
  \[
  X = \lim_{\longleftarrow} \left(   
  K_1\xleftarrow{p_1}K_2\xleftarrow{p_2}\cdots \right)
  \]
  Let $v(K_i)$ denote the set of vertices of $K_i$.
  We let
  \[
    \B{K_i} = \{ \bst_{K_i} v \}_{v \in v(K_i)}
  \]
  be the cover of $K_i$ by barycentric stars of vertices of $K_i$ (see Definition~\ref{def:stars}) and
  \[
    \B{i} = \{ \pi^{-1}_i(\bst_{K_i} v) \}_{v \in v(K_i)}
  \]
  be the cover of $X$ by sets of threads that pass through elements of $\B{K_i}$.
\end{definition}

\begin{definition}
  We say that a cover $\F{} = \{ F_i \}_{i \in I}$ is \df{isomorphic} to a
  cover $\G{} = \{ G_i \}_{i \in I}$ if for each $J \subset I$ we have
  \[
    \bigcap_{j \in J} F_j \neq \emptyset \iff 
    \bigcap_{j \in J} G_j \neq \emptyset.
  \]
  Note the identical indexing set of $\F{}$ and $\G{}$.
\end{definition}
\begin{definition}
  Let $p \colon Y \to Z$ be a map.
  Let $\F{} = \{ F_i \}_{i \in I}$ be a cover of $Z$.
  A \df{pull-back of $\F{}$} is a cover $p^{-1}(\F{})$ of $Y$ defined by the formula
  \[
  p^{-1}(\F{}) = \{ p^{-1}(F_i) \}_{i \in I}.
  \]
  The pull-back $p^{-1}(\F{})$ retains the indexing set $I$ of $\F{}$.
\end{definition}

\begin{lemma}\label{lem:pullback properties}
  Let $f \colon X \to Y$ be a map, let $\G{}$ be a cover of $Y$ and let $f^{-1}(\G{})$ be a pull-back of $\G{}$.
  \begin{enumerate}
      \item If $f$ is surjective, then the covers $f^{-1}(\G{})$ and $\G{}$ are isomorphic.
      \item If $\G{}$ is open / closed / locally finite / locally finite-dimensional, then so is $f^{-1}(\G{})$.
  \end{enumerate}
\end{lemma}

\section{Nerve Theorem}

A Nerve Theorem in its abstract form states that if two spaces admit isomorphic $AE(n)$-covers, then they are weak $n$-homotopy equivalent.
For general spaces, some local finiteness and dimension restrictions are placed on the covers~\cite{nagorko2007}.
In our case we need such a theorem without these restrictions, which we are able to prove for polyhedra covered by subcomplexes.

The goal of this section is to prove the following theorem.

\begin{theorem}\label{thm:nregular pullback}
  If a quasi-simplicial map $p \colon K \to L$ of polyhedra is surjective and $n$-regular, then it is a weak $n$-homotopy equivalence and the pull-back $p^{-1}(\B{L})$ is an $AE(n)$-cover of $K$.
\end{theorem}

The main tool used in the proof is Theorem~\ref{thm:nerve theorem}, which when applied to a canonical map into the nerve of a cover~\cite{nagorko2007} implies the usual Nerve Theorem.

\subsection{Nerve Theorem for non locally finite covers}

\begin{lemma}\label{lem:deformation}
  If $K$ is a full subpolyhedron of a polyhedron $L$, then 
  \begin{enumerate}
  \item $K$ is a deformation retract of $\bst_L K$; and
  \item $K$ is a deformation retract of $\ost_L K$.
  \end{enumerate}
\end{lemma}
\begin{proof}
  We let $V_L$ denote the set of vertices of $L$ and $V_K \subset V_L$ denote the set of vertices of $K$.
  We regard $L$ as a subspace of $\ell_1(V_L)$ (see Definition~\ref{def:polyhedron metric}).
  Let $p \colon L \to \ell_1(V_L)$ be the map that sends all coordinates that do not belong to $K$ to $0$, defined by the formula
  \[
  \left( p(x) \right)_v =
  \left\{
  \begin{array}{ll}
  0 & v \in V_L \setminus V_K \\
  (x)_v & v \in V_K
  \end{array}
  \right..
  \]
  
  We have
  \[
    \ost_L K = \{ x \in L \colon p(x) \neq 0 \}.
  \]
  Let $q \colon \ost_L K \to L$ be defined by the formula
  \[
    q(x) = \frac{p(x)}{\| p(x) \|_1}.
  \]
  We define a map $\varPhi \colon \ost_L K \times [0, 1] \to L$ by the formula
  \[
  \varPhi(x, t) = t \cdot q(x) + (1-t)\cdot x\text{.}
  \]
  Observe that $\supp \varPhi(x, t) \subset \supp(p(x))$ and $V_K \cap \supp \varPhi(x, t) \neq \emptyset$, hence $\varPhi(x, t) \in \ost_L K$. 
  Note that $\varPhi(x, 1) \in K$ because $K$ is a full subcomplex.
  Hence $\varPhi$ is a deformation retraction of $\ost_L K$ to $K$. 
  
  Observe that
  \[
    \bst_L K = \{ x \in L \colon \exists_{v \in V_K} (x)_v \geq \max \{ (x)_w \colon w \in V_L \} \}.
  \]
  Since $\varPhi$ preserves the relation $\exists_{v \in V_K} (x)_v \geq \max \{ (x)_w \colon w \in V_L \}$, $\varPhi$ restricted to $\bst_L K \times [0,1]$ is into $\bst_L K$, hence it is a deformation retraction of $\bst_L K$ to $K$.
\end{proof}

\begin{corollary}\label{cor:open star swelling}
  If $K$ is a polyhedron and $\mathcal{K} = \{ K_i \}_{i \in I}$ is an $AE(n)$-cover of $K$ by subcomplexes, then the collection
  \[
    \ost_{\beta K} \mathcal{K} = \{ \ost_{\beta K} K_i \}_{i \in I}
  \]
  is an open $AE(n)$-cover of $K$.
  Moreover, $\ost_{\beta K} \mathcal{K}$ and $\mathcal{K}$ are isomorphic covers of $K$.
\end{corollary}
\begin{proof}
  Observe that if $\{ A_j \}_{j \in J}$ is a collection of subcomplexes of a polyhedron~$K$, then
  \[
     \bigcap_{j \in J} \ost_{\beta K} A_j = 
       \ost_{\beta K} \bigcap_{j \in J} A_j.
  \]
  An application of Lemma~\ref{lem:deformation} finishes the proof.
\end{proof}

\begin{lemma}\label{lem:uclose uhomotopic}
  If $\U{}$ is an open $AE(n)$-cover of a space $Y$,
    then any two $\U{}$-close maps from a metric space of dimension less than $n$ into $Y$ are $\U{}$-homotopic.
\end{lemma}
\begin{proof}
  Let $\U{}$ be an open $AE(n)$-cover of a metric space $Y$.
  Let $X$ be a metric space of dimension less than $n$.
  Let $f, g \colon X \to Y$ be $\U{}$-close.
  We have to show that $f$ and $g$ are homotopic
    by a homotopy whose paths refine $\U{}$.
  
  Let $\V{} = \{ V_U \}_{U \in \U{}}$ be the collection of subsets of $X$ defined by $V_U = f^{-1}(U) \cap g^{-1}(U)$. Since $f$ and $g$ are $\U{}$-close, $\V{}$ is an open cover of $X$.
  
  Let $\W{}$ be a closed, locally finite cover of $X$ with multiplicity at most $n$ that refines $\V{}$.
  Let $\F{} = \{ W \times [0,1] \colon W \in \W{} \}$ be a cover of $X \times [0,1]$.
  Note that $\F{}$ is a closed, locally finite, locally finite-dimensional cover of $X \times [0,1]$.
  
  Let $H_0 \colon K \times \{ 0, 1 \} \to Y$ be the map defined by 
    $H_0(x, 0) = f(x)$ and $H_0(x,1) = g(x)$.

  Let $C \colon \V{} \to \U{}$ be the map defined by $C(V_U) = U$. It is a carrier and both $f$ and $g$ are carried by $C$, directly from the definition of $\V{}$.
  Let $D \colon \W{} \to \V{}$ be a map such that for each $W \in \W{}$ we have $W \subset D(W)$. The map $D$ exists because $\W{}$ refines $\V{}$.
  Let $E \colon \F{} \to \W{}$ be a map defined by $D(W \times [0,1]) = W$.
  The composition $C \circ D \circ E$ is a carrier and $H_0$ is carried by it.
  Observe that $\U{}$ is an $AE(X \times [0,1])$-cover of $Y$~(see Definition~\ref{def:AE(C)}).
  By the Carrier Theorem, $H_0$ can be extended to a map $H \colon X \times [0,1]$ that is carried by $C \circ D \circ E$.
  Clearly, $H$ is a homotopy between $f$ and $g$.
  For each path $\{ x \} \times [0, 1] \subset X \times [0,1]$ there exists element $W \times [0, 1] \in \F{}$ such that $\{ x \} \times [0, 1] \subset W \times [0, 1]$, as $\W{}$ is a cover of $X$. Since $H$ is carried into $\U{}$, the whole path lies in an element of $\U{}$. Therefore $H$ is a $\U{}$-homotopy.
\end{proof}

\begin{theorem}\label{thm:nerve theorem}
  Let $\mathcal{K} = \{ K_i \}_{i \in I}$ be a cover of a polyhedron $K$ by subcomplexes.
  Let $\mathcal{L} = \{ L_i \}_{i \in I}$ be a cover of a polyhedron $L$ by subcomplexes.
  Let $p \colon K \to L$ be a surjective simplicial map that maps elements of~$\mathcal{K}$ into the corresponding elements of $\mathcal{L}$ (i.e. $p(K_i) \subset L_i$ for each $i \in I$).
  If $\mathcal{K}$ and $\mathcal{L}$ are isomorphic $AE(n)$-covers, then $p$ is a weak $n$-homotopy equivalence.
\end{theorem}

\begin{proof}
  To begin, we will show that $p$ induces monomorphisms on homotopy groups of dimensions less than $n$.
  Let $m < n$. Let $\varphi \colon S^m \to K$. 
  Assume that $p \circ \varphi$ is null-homotopic in $L$. 
  Let $\varPhi \colon B^{m+1} \to L$ denote such a null-homotopy ($\varPhi_{\mid S^m} = p \circ \varphi$).
  We have to show that $\varphi$ is null-homotopic in $K$.

  Let $\F{}$ be a finite closed cover of $B^{m+1}$ with mesh small enough so that for each $F \in \F{}$ we can pick $i_F \in I$ such that
  \[
    (\ast)\  \varphi(F \cap S^m) \subset \ost_{\beta K} K_{i_F} \text{ and } 
    (\ast\ast)\  \varPhi(F) \subset \ost_{\beta L} L_{i_F}.
  \]
  By $(\ast\ast)$, a map $C \colon \mathcal{F} \to \ost_{\beta L} \mathcal{L}$ defined by $C(F) = \ost_{\beta L} L_{i_F}$ is a carrier.
  Since $\mathcal{K}$ and $\mathcal{L}$ are isomorphic, the map $C' \colon \mathcal{F} \to \ost_{\beta K} \mathcal{K}$ defined by $C'(F) = \ost_{\beta K} K_{i_F}$ is a carrier as well.
  By $(\ast)$, $\varphi$ is carried by $C'$.
  By Corollary~\ref{cor:open star swelling}, $\ost_{\beta K} \mathcal{K}$ is an $AE(n)$-cover.
  By the Carrier Theorem, $\varphi$ extends to a map $\tilde \varphi \colon B^{m+1} \to K$ that is carried by $C'$.
  This is a null-homotopy of $\varphi$ in $K$ and we are done with the proof that $p$ induces monomorphisms on homotopy groups of dimensions less than $n$.
  
  Next we will show that $p$ induces epimorphisms on homotopy groups of dimensions less than $n$.
Let $m < n$.
Let $\psi \colon S^m \to L$.
Let $\G{}$ be a closed finite cover of $S^m$ that refines $\psi^{-1}(\ost_{\beta L} \mathcal{L})$.
For each $G \in \G{}$ pick $i_G \in I$ such that $\psi(G) \subset \ost_{\beta L} L_{i_G}$.
Then the map $D \colon \G{} \to \ost_{\beta L} \mathcal{L}$ defined by $D(G) = \ost_{\beta L} L_{i_G}$ is a carrier and $\psi$ is carried by $D$. 
Since $\mathcal{K}$ and $\mathcal{L}$ are isomorphic, the map $D'(G) = \ost_{\beta K} K_{i_G}$ is a carrier as well.
By Corollary~\ref{cor:open star swelling}, $\ost_{\beta K} \mathcal{K}$ is an $AE(n)$-cover.
By the Carrier Theorem, there exists a map $\tilde \psi \colon S^m \to K$ that is carried by $D'$.
Observe that if $x \in G \in \G{}$, then $\tilde \psi(x) \in \ost_{\beta_K} K_{i_G}$ and $\psi(x) \in \ost_{\beta L} L_{i_G}$.
Since $p(K_{i_G}) \subset L_{i_G}$, the maps $\psi$ and $p \circ \tilde \psi$ are $\ost_{\beta L} \mathcal{L}$-close.
By Corollary~\ref{cor:open star swelling}, $\ost_{\beta L} \mathcal{L}$ is an $AE(n)$-cover and by
 Lemma~\ref{lem:uclose uhomotopic}, $\psi$ and $p \circ \tilde \psi$ are homotopic.
Hence $p$ induces epimorphisms on homotopy groups of dimensions less than $n$. This concludes the second part of the proof.
\end{proof}

\subsection{Proof of Theorem~\ref{thm:nregular pullback}}

\begin{lemma}\label{lem:simplicial regular}
  Let $K$ and $L$ be polyhedra.
  If $p \colon K \to L$ is a simplicial, surjective map such that for each $\delta \in \tau(L)$ the inverse image $p^{-1}(\delta)$ is an $AE(n)$, then $p$ is a weak $n$-homotopy equivalence.
\end{lemma}
\begin{proof}
  Apply Theorem~\ref{thm:nerve theorem} with the cover $\mathcal{L} = \tau(L)$ of $L$ and the cover $\mathcal{K} = p^{-1}(\mathcal{L})$ of $K$.
\end{proof}

\begin{proof} (Proof of Theorem~\ref{thm:nregular pullback})
  The map $p$ is simplicial onto $\beta K$ and satisfies the conditions of Lemma~\ref{lem:simplicial regular}, hence it is a weak $n$-homotopy equivalence.
  
  Let $\mathcal{A} \subset \B{L}$ such that $\bigcap_{A \in \mathcal{A}} A \neq \emptyset$.
  By Lemma~\ref{lem:stars}, $\bigcap_{A \in \mathcal{A}} A$ is an absolute extensor in dimension $n$.
  We have $B = \bigcap_{A \in \mathcal{A}} p^{-1}(A) = p^{-1}(\bigcap_{A \in \mathcal{A}} A)$.
  The map $p_{|B} \colon B \to \bigcap_{A \in \mathcal{A}} A$ satisfies the conditions of Lemma~\ref{lem:simplicial regular}, hence it is a weak $n$-homotopy equivalence.
  Therefore $B$ has vanishing homotopy groups in dimensions less than $n$ and by Theorem~\ref{thm:dugundji}, it is an absolute extensor in dimension $n$. 
  Hence $p^{-1}(\B{L})$ is an $AE(n)$-cover.
\end{proof}

\section{A lifting condition}


\begin{lemma}\label{lem:single lift}
  Let $p \colon Y \to Z$ be a surjective map of metric spaces.
  Let~$\F{}$ be a cover of $Z$.
  Assume that either $\F{}$ is an open cover or $\F{}$ is a
    closed, locally finite, locally finite-dimensional cover.
  Assume that $p^{-1}(\F{})$ is an $AE(n)$-cover.

  Let $X$ be an at most $n$-dimensional metric space.
  Let $f \colon X \to Z$.
  Let $A$ be a closed subset of $X$.
  Let $g_0 \colon A \to Y$ be a map such that $p \circ g_0 = f_{|A}$ (a lift of $f$ on $A$).
  
  Then there exists a map $g \colon X \to Y$ that 
  satisfies the following conditions:
  \begin{enumerate}
  \item $g_{|A} = g_0$ ($g$ extends $g_0$); and
  \item $p\circ g$ is $\F{}$-close to $f$.
  \end{enumerate}
\end{lemma}
\begin{proof}
  We have the following commutative diagram.
  \[
    \xymatrix@M=8pt@C=35pt{
    Y \ar[r]^{p} & Z  \\
    A \ar@{^{(}->}[r] \ar[u]_{g_0} & X \ar[u]^{f} \ar@{.>}[ul]_{g} \\
    }
  \]
  Let $\mathcal{P} = p^{-1}(\F{})$ and $\G{} = f^{-1}(\F{})$.
  Let $C \colon \G{} \to \mathcal{P}$ be defined by the formula $C(f^{-1}(F)) = p^{-1}(F)$. 
  Since $p$ is surjective, $C$ is a carrier. 
  For each $F \in \F{}$ and each $x \in f^{-1}(F) \cap A$ we have $g_0(x) \in p^{-1}(f(x)) \in p^{-1}(F)$; hence $g_0$ is carried by $C$.
  
  Let $\mathcal{H}$ be a closed locally finite locally finite-dimensional cover of $X$ such that $\mathcal{H}$ refines $\G{}$. 
  If $\F{}$ is an open cover, then $\mathcal{H}$ exists as a refinement of an open cover $\G{}$.
  If $\F{}$ is closed locally finite locally finite-dimensional, then we can take $\mathcal{H} = \G{}$, which has the required properties by Lemma~\ref{lem:pullback properties}.
  Let $D \colon \mathcal{H} \to \G{}$ be any map such that $H \subset D(H)$ for each $H \in \mathcal{H}$.
  
  Observe that $C \circ D \colon \mathcal{H} \to \mathcal{P}$ is a carrier, $g_0$ is carried by $C \circ D$, $\mathcal{P}$ is an $AE(n)$-cover, $X$ is at most $n$-dimensional metric space and $\mathcal{H}$ is closed, locally finite and locally finite-dimensional. By the Carrier Theorem, there exists a map $g \colon X \to Y$ that extends $g_0$ and that is carried by $C \circ D$.
  This implies that for each $x \in X$ there exists $F \in \F{}$ such that $x \in f^{-1}(F)$ and $g(x) \subset p^{-1}(F)$. 
  Therefore $p(g(x)) \in F$ so $p \circ g$ is $\F{}$-close to $f$.
\end{proof}

\subsection{A lifting condition for inverse limits}

\begin{numbered_definition}\label{def:limit}
Let
\[
 Z=\lim_{\longleftarrow}\left( Z_1\xleftarrow{p_1}Z_2\xleftarrow{p_2}\cdots \right).
\]
Let $\F{i}$ be a cover of $Z_i$.
We define the following conditions.
\begin{enumerate}[label=(\Alph*)]
\item For each $i$, $Z_i$ is a complete metric space.
\item For each $i$, the bonding map $p_i \colon Z_{i+1} \to Z_i$ is surjective and $1$-Lipschitz.
\item For each $i$, $\F{i}$ is either an open cover or a closed, locally finite, locally finite-dimensional cover.
\item For each $i$, the pull-back $p_i^{-1}(\F{i})$ is an $AE(n)$-cover.
\item $\sum_i \mesh \F{i} < \infty$.
\end{enumerate}

We denote short projections in the inverse limit by $\pi^k_i \colon Z_k \to Z_i (k > i)$ and long projections in the inverse limit by $\pi_k \colon Z \to Z_k$.
\end{numbered_definition}

\begin{lemma}\label{lem:limit lift}
  Let
\[
 Z=\lim_{\longleftarrow}\left( Z_1\xleftarrow{p_1}Z_2\xleftarrow{p_2}\cdots \right).
\]
  Let $\F{i}$ be a cover of $Z_i$. 
  Assume that conditions of Definition~\ref{def:limit} are satisfied.
  
  Let $X$ be an at most $n$-dimensional metric space.
  Let $f \colon X \to Z_1$.
  Let $A$ be a closed subset of $X$.
  Let $g_0 \colon A \to Z$ be a map such that $\pi_1 \circ g_0 = f_{|A}$ (a lift of $f$ on $A$).
  
  Then there exists a map $g \colon X \to Z$ such that $g_{|A} = g_0$ ($g$ extends $g_0$).
\end{lemma}
\begin{proof}
Let $f_1 = f$. 
Applying Lemma~\ref{lem:single lift} for each $i>1$ we construct a map
  $f_i \colon X \to Z_i$ that satisfies the following conditions.
\begin{enumerate}[label=(\roman*)]
  \item $p_{i-1} \circ f_{i}$ is $\F{i-1}$-close to $f_{i-1}$.
  \item $\pi_i \circ g_0 = f_i\mid_A$.
\end{enumerate}
For each $k$ and each $m > k$ we let 
\[
a_m^k=\pi^m_k\circ f_m: X \to Z_k. 
\]
By (i), we have $\dsup(f_m, p_m \circ f_{m+1}) \leq \mesh \F{m}$.
By (B) the short projection $\pi^m_k$ is $1$-Lipschitz, hence $\dsup(a^k_m, a^k_{m+1})
  = \dsup(\pi^m_k \circ f_m, \pi^m_k \circ p_m \circ f_{m+1}) \leq \mesh \F{m}$.
Therefore for $l > m$ we have
\[
  \dsup (a^k_l, a^k_m) \leq \dsup(a^k_l, a^k_{l-1}) + \dsup(a^k_{l-1}, a^k_{l-2})
    + \cdots + \dsup(a^k_{m+1}, a^k_m) \leq 
\]
\[
  \leq \mesh \F{l-1} + \mesh \F{l-2} + \cdots + \mesh \F{m+1} + \mesh \F{m}.
\]
By (E), the sequence $a_m^k$ is uniformly convergent. Let
\[
  a^k = \lim_{m \to \infty} a_m^k.
\]
By (A), $Z_k$ is complete, hence $a^k \colon X \to Z_k$ is well defined.

It follows from the definition that $p_k \circ a^{k+1}_m = a^k_m$. Passing to the limit we have 
\[
  p_k \circ a^{k+1} = \lim_{m \to \infty} p_k \circ a^{k+1}_m = \lim_{m \to\infty} a^k_m = a^k.
\]
Therefore, for each $x$ the sequence $(a^k(x))_k$ is a thread in $Z$ and
we can define a map $g \colon X \to Z$ by the formula
\[
  (g(x))_k = a^k(x).
\]

Observe that by (ii) and by the definition of $a^k_m$ we have
\[
a^k_m |_A = \pi^m_k \circ f_m |_A = \pi^m_k \circ \pi_m \circ g_0  |_A = \pi_k \circ g_0 |_A.
\]
Therefore $a^k\mid_A = \pi_k \circ g_0$ for each $k$, hence $g |_A = g_0 |_A$.
\end{proof}

\begin{theorem} \label{thm:ane condition} 
  Let
\[
 Z=\lim_{\longleftarrow}\left( Z_1\xleftarrow{p_1}Z_2\xleftarrow{p_2}\cdots \right).
\]
  Let $\F{i}$ be a cover of $Z_i$. 
  If the conditions of Definition~\ref{def:limit} are satisfied and $Z_1$ is an $ANE(n)$, then $Z$ is an $ANE(n)$.
\end{theorem}
\begin{proof}
  We will show that $Z$ is an $ANE(n)$ directly from the definition.
  Let $X$ be an at most $n$-dimensional metric space and let $A\subset X$ be a closed subset. Take $g_0\colon A\to Z$. By assumption, $Z_1$ is an $ANE(n)$ so we can extend $\pi_1\circ g_0:A\to Z_1$ to a map $f_1:U\to Z_1$, where $U$ is an open neighborhood of $A$ in $X$. 
  By Lemma~\ref{lem:limit lift}, $f_1$ can be lifted to a map $g \colon U \to Z$ such that $g_{|A} = g_0$. This is the extension of $f$ we sought.
\end{proof}

\section{Local $k$-connectedness of inverse limits of polyhedra}

\begin{definition}
  Let
  \begin{equation}\label{eq:sequence}\tag{L}
    X = \lim_{\longleftarrow}\left( K_1 \xleftarrow{p_1} K_2 \xleftarrow{p_2} \cdots.\right)
  \end{equation}
  Fix $m$ and let $A \subset K_m$ be a subcomplex of $K_m$.
  A \df{restriction of \eqref{eq:sequence} to $A$} is the inverse limit
  \begin{equation}\label{eq:restriction}\tag{R}
    X' = \lim_{\longleftarrow} \left(K'_m \xleftarrow{p'_m} K'_{m+1} \xleftarrow{p'_{m+1}} \cdots\right),
  \end{equation}
  where $K'_m =\ A$ and for each $j \geq m$, $K'_{j+1} = p^{-1}_j(K'_j)$ and $p'_j = p_j \mid_{K'_{j+1}}$.
\end{definition}
\begin{lemma}\label{lem:restriction}
  Let \eqref{eq:restriction} be the restriction of~\eqref{eq:sequence} to a subcomplex $A \subset K_m$.
  
  If, for each $i$, \eqref{eq:sequence} satisfies the conditions:
  \begin{enumerate}[label=(\Roman*)]
  \item $K_i$ is a polyhedron and
  \item $p_i$ is a quasi-simplicial map that is surjective and $n$-regular,
  \end{enumerate}
  then so does \eqref{eq:restriction}.
  The inverse limit $X'$ is homeomorphic to $\pi_m^{-1}(A)$, where $\pi_m \colon X \to X_m$ denotes the long projection.
\end{lemma}
\begin{proof}
  We have
  \[
    \pi^{-1}_m(A) = \lim_{\longleftarrow} \left(K'_1 \xleftarrow{p'_1} K'_{2} \xleftarrow{p'_{2}} \cdots\right),
  \]
  where $K'_j = p_j(K'_{j+1})$ and $p'_j = p_j \mid_{K'_{j+1}}$ for $j < m$.
  The restriction \eqref{eq:restriction} is the same sequence with first $m-1$ elements removed. This changes the metric on the limit, but not the topology, hence $X'$ is homeomorphic to $\pi_m^{-1}(A)$.
  
  The other conditions are trivial to verify.
\end{proof}

\begin{theorem}\label{thm:ane limit of polyhedra}
  Let
  \[
    X = \lim_{\longleftarrow} \left(K_1 \xleftarrow{p_1} K_2 \xleftarrow{p_2} \cdots\right).
  \]
  Assume that for each $i$ the following conditions are satisfied:
  \begin{enumerate}[label=(\Roman*)]
  \item $K_i$ is a locally finite-dimensional polyhedron; and
  \item $p_i$ is a quasi-simplicial map that is surjective and $n$-regular.
  \end{enumerate}
  Then 
  \begin{enumerate}
  \item $X$ is an $ANE(n)$;
  \item each short projection $\pi^k_i \colon K_k \to K_i$ and each long projection $\pi_i \colon X \to X_i$ is a weak $n$-homotopy equivalence;
  \item for each $i$, the covers $\O{i}$ and $\B{i}$ are $AE(n)$-covers of $X$.
  \end{enumerate}  
\end{theorem}
\begin{proof}
  Fix a 
  metric of scale $2^{-i}$ on $K_i$ (see Definition~\ref{def:polyhedron metric}). 
  Let $(*)$ denote the inverse limit
  \[
    X = \lim_{\longleftarrow} \left( K_1 \xleftarrow{p_1} K_2 \xleftarrow{p_2} \cdots\right)
  \]
  along with a sequence $\B{K_i}$ of covers, where $\B{K_i}$ is a cover of $K_i$ by barycentric stars of vertices.
  
  We will verify that $(*)$ satisfies conditions (A)--(E) of Definition~\ref{def:limit}.
  
  Condition (A): By Lemma~\ref{lem:complex is ane}, every $K_i$ is a complete metric space.
  
  Condition (B): By condition (II), $p_i$ is surjective. By the choice of scale on $K_i$ and Lemma~\ref{lem:qs is lipschitz}, each bonding map is $1$-Lipschitz.
  
  Condition (C): By Lemma~\ref{lem:stars}, each $\B{K_i}$ is a closed, locally finite, locally finite-dimensional cover. 
  
  Condition (D): By assumption $p_i$ is quasi-simplicial and $n$-regular. Hence by Theorem~\ref{thm:nregular pullback}, the pull-back $p_i^{-1}(\B{K_i})$ is an $AE(n)$-cover.
  
  Condition (E): By the choice of the metric on $K_i$, we have $\sum_i \mesh \B{K_i} \leq \sum_i \diam K_i < \infty$.
  
  To prove (1) it is enough to verify the conditions of Theorem~\ref{thm:ane condition} for $(*)$.
  We just verified conditions of Definition~\ref{def:limit}.
  By Lemma~\ref{lem:complex is ane}, $K_1$ is an $ANE(n)$.
  We are done.
  
  By Theorem~\ref{thm:nregular pullback} each $p_i$ is a weak $n$-homotopy equivalence and therefore all the short projections $\pi^k_i$ are weak $n$-homotopy equivalences. To finish the proof of (2), we must show that 
  \begin{itemize}
      \item[(mono)] for fixed $i>0$ and $m<n$, the long projection $\pi_i:X\to K_i$ induces a monomorphism on the homotopy group of dimension $m$, regardless of the choice of base point; and
      \item[(epi)] for fixed $i>0$ and $m<n$, the long projection $\pi_i:X\to K_i$ induces an epimorphism on the homotopy group of dimension $m$, regardless of the choice of base point.
  \end{itemize}

  We prove (mono). Fix $i > 0$ and $m < n$. 
  Let $\varphi \colon S^m \to X$.
  Assume that $\pi_i \circ \varphi$ is null-homotopic in $K_i$.
  Let $\varPhi \colon B^{m+1} \to K_i$ denote the null-homotopy ($B^{m+1}$ denotes the $m+1$-dimensional unit ball in $\mathbb{R}^{m+1}$).
  We have the following commutative diagram:
  \[
    \xymatrix@M=8pt@C=35pt{
    X \ar[r]^{\pi_i} & K_i  \\
    S^m \ar@{^{(}->}[r] \ar[u]_{\varphi} & B^{m+1} \ar[u]^{\varPhi} \ar@{.>}[ul]_{\tilde \varPhi} \\
    }
  \]
  The inverse limit
  \[
    X = \lim_{\longleftarrow} \left(K_i \xleftarrow{p_i} K_{i+1} \xleftarrow{p_{i+1}} \cdots \right)
  \]
  along with the sequence of covers by barycentric stars of vertices (which we obtain by truncating (*)) satisfies the conditions of Definition~\ref{def:limit}.
  Since $B^{m+1}$ is at most $n$-dimensional, by Lemma~\ref{lem:limit lift}, there exists a lift $\tilde \varPhi$ such that the diagram is commutative. 
  This lift is a null-homotopy of $\varphi$ in $X$, hence $\pi_i$ induces a monomorphism on the homotopy group of dimension $m$.
  We will show (epi) 
  at the end of the proof.
  
  Next we prove (3). To begin, we will show that for each $i$, $\B{i}$ is an $AE(n)$-cover.
  Let $\mathcal{A}$ be a collection of elements of $\B{i}$ such that the intersection $\bigcap \mathcal{A}$ is non-empty.
  By the definition of $\B{i}$, we have $\mathcal{A} = \{ \pi_i^{-1}(\bst_{K_i} v) \}_{v \in V}$ for some set of vertices $V$ of $K_i$. 
  Let $A = \bigcap_{v \in V} p_i^{-1}(\bst_{K_i} v)$.
  Since $p_i$ is quasi-simplicial, $A$ is a subcomplex of $K_{i+1}$.
  By Theorem~\ref{thm:nregular pullback}, $A$ has vanishing homotopy groups in dimensions less than $n$.
  Let
  \begin{equation}\tag{**}
    X' = \lim_{\longleftarrow} \left( K'_{i+1} \xleftarrow{p'_{i+1}} K'_{i+2} \xleftarrow{p'_{i+2}} \cdots\right),
  \end{equation}
  where $K'_{i+1} =\ A$ and for each $j \geq i+1$, $K'_{j+1} = p^{-1}_j(K'_j)$ and $p'_j = p_j \mid_{K'_{j+1}}$, be the restriction of (*) to $A$.
  By Lemma~\ref{lem:restriction} it satisfies assumptions (I) and (II) and is homeomorphic to $\bigcap \mathcal{A}$. Hence from what we have already proven, $X' = \bigcap \mathcal{A}$ is an $ANE(n)$ and the long projection $\pi_{i+1} \colon \bigcap \mathcal{A} \to A$ induces monomorphisms on homotopy groups of dimensions less than $n$. Since $A$ has vanishing homotopy groups in these dimensions, so does $\bigcap \mathcal{A}$. By Theorem~\ref{thm:dugundji}, $\bigcap \mathcal{A}$ is an $AE(n)$ hence $\B{i}$ is an $AE(n)$-cover. 
  
  It follows from Lemma~\ref{lem:union is ae} that $\O{i}$ is an $AE(n)$-cover, as open stars are (infinite) unions of iterated barycentric stars.
  
  
  Finally, we prove (epi). 
  Fix $i > 0$ and $m < n$. 
  Let $\varphi \colon S^m \to K_i$.
  Let $\mathcal{O}$ be a cover of $K_i$ by open stars of vertices.
  We have just shown that $\pi_i^{-1}(\mathcal{O})$ is an $AE(n)$-cover.
  Hence by Lemma~\ref{lem:single lift}, there exists a map $\psi \colon S^m \to X$ such that $\pi_i \circ \psi$ and $\varphi$ are $\mathcal{O}$-close.
  By Lemma~\ref{lem:uclose uhomotopic}, these maps are homotopic.
  This shows that $\pi_i$ induces epimorphisms on homotopy groups of dimensions less than $n$.
  Hence $\pi_i$ is a weak $n$-homotopy equivalence.
\end{proof}

\begin{corollary}
  Let
  \[
    X = \lim_{\longleftarrow} \left(K_1 \xleftarrow{p_1} K_2 \xleftarrow{p_2} \cdots\right).
  \]
  Assume that for each $i$, (I) $K_i$ is a locally finite-dimensional polyhedron; (II) $p_i$ is a quasi-simplicial map that is surjective and $n$-regular.
  Let $Y$ be a metric space of dimension less than $n$ and let $f, g \colon Y \to X$.
  Then if $f$ is $\O{i}$-close to $g$, then $f$ is $\O{i}$-homotopic to $g$.
\end{corollary}
\begin{proof}
  By Theorem~\ref{thm:ane limit of polyhedra}(3), $\O{i}$ is an $AE(n)$-cover of $X$. By Lemma~\ref{lem:uclose uhomotopic}, $f$ and $g$ are $\O{i}$-homotopic.
\end{proof}

\bibliographystyle{abbrv}
\bibliography{references2}

\end{document}